\documentclass{article}%
\usepackage{amsmath}
\usepackage{amsfonts}
\usepackage{amssymb, color}
\usepackage{graphicx}%
\usepackage{hyperref}
\providecommand{\U}[1]{\protect \rule{.1in}{.1in}}
\newtheorem{theorem}{Theorem}

\newtheorem{definition}[theorem]{Definition}

\newtheorem{lemma}[theorem]{Lemma}

\newtheorem{proposition}[theorem]{Proposition}
\newtheorem{remark}[theorem]{Remark}

\newenvironment{proof}[1][Proof]{\noindent \textbf{#1.} }{\  \rule{0.5em}{0.5em}}

\newcommand{\hmE}{{ \mathbb {\hat E}}}
\newcommand{\lmu}{\underline{\mu}}
\newcommand{\bmu}{\bar{\mu}}
\newcommand{\R}{{\mathbb R}}
\begin{document}

\title{  Optimal  Unbiased Estimation  for Maximal Distribution }
 \date{November 24, 2016}
 \author{Hanqing Jin\thanks{Mathematical Institute, Andrew Wiles Building, Oxford University}\,\,\,\,  and \,\,\,\,  Shige Peng\thanks{
School of Mathematics and Institute of Advanced Research, Shandong University, Jinan 250100, China. This research is partially supported by 
Zhongtai Institute of Finance, Shandong University and NSF China, No. L1624032. and No. 11526205 and 111 Project, No. B12023 of Chinese SAFEA.}}

 \maketitle
 
 \begin{abstract}
 Unbiased estimation for parameters of  maximal distribution is a very fundamental problem in the statistical theory of sublinear expectation. In this paper, we  proved that the maximum estimator is the largest unbiased estimator for the upper mean and the minimum estimator is the smallest unbiased estimator the the lower mean.   \\
 
{\bf Keywords: } sublinear expectation,  maximal distributions, Optimal unbiased estimation.   \end{abstract}

\section{Introduction}

Motivated by model uncertainties in statistics, the risk measures, super hedge pricing and modeling uncertainty in finance, Peng \cite{p1, p2,p3,p4, p5, p6} initiated the study of nonlinear expectations.   Since it has essentially and even naturally embodied the uncertainty occurred in probability models and distribution models, the theory of nonlinear expectation has been developed quickly and has already made substantial progress in recent years. However,   statistical analysis in nonlinear expectation cases hasn't accomplished such development. We refer to Lin et al. \cite{Lin1,Lin2} and Rokhlin \cite{Rokhlin}, up to our knowledge. One of the main reasons is that it takes time to fully comprehend the meaning of the theoretical results and data analysis.  

In the classical statistics theory, the law of large numbers (LLN) provide the basis for obtaining the unbiased estimator for parameters of unknown distribution. Peng \cite{p3,p5} obtained the  corresponding LLN  in the setting of sublinear expectations.    For any
i.i.d. sequence $\{ \xi_{i}\}_{i=1}^{\infty}$ in a sublinear expectation space
$(\Omega,\mathcal{H},\mathbb{\hat{E}})$ satisfying
\[
\lim_{c\rightarrow \infty}\mathbb{\hat{E}}[(|\xi_{1}|-c)^{+}]=0,
\]
then $\frac{1}{n}S_{n}=\frac{1}{n}(\xi_{1}+\cdots+\xi_{n})$ converges in law to
a maximal distribution $X\overset{d}{=}M_{[\underline{\mu},\overline{\mu}]}$. (Some details will be presented in next section).  The parameters $\underline{\mu}$ and $\overline{\mu}$ describe the distributional uncertainty. The problem of how to estimate the
parameters $\underline{\mu}$ and $\overline{\mu}$ becomes also crucially
important. In fact this problem also concerns how the whole theory of
nonlinear expectation can be applied with it's statistics.  In this paper,  we will discuss a very fundamental problem: how to choose suitable  estimators for the upper mean
$\overline{\mu}=\mathbb{\hat{E}}[X]$ and lower-mean $\underline{\mu
}=-\mathbb{\hat{E}}[-X_{1}]$ for a maximal distributed random variable $X$.  We will prove that the maximum
estimator  is the largest unbiased  estimator for
$\overline{\mu}$ and  the minimum
estimator  is the smallest unbiased
estimator for $\underline{\mu}$.  At last, we use this result to provide a very
general estimators for the nonlinear distribution of $X$.

The paper is organized as follows: in Section 2, we present some basic
theoretical result of sublinear expectation and the notion of G-distributions.
The estimators of maximally distributed random variables is studied in Section 3.   A general estimator for the distribution of a general random variable  will be given  in Section 4.  

\section{Preliminary}
\subsection{ Sublinear expectations}
Let $\Omega$ be a given set and let $\mathcal{H}$ \label{huah}be a linear
space of real valued functions defined on $\Omega$. In this paper, we suppose
that $\mathcal{H}$ satisfies $c\in\mathcal{H}$ for each constant $c$ and $%
|X|\in\mathcal{H}$ if $X\in\mathcal{H}$. The space $\mathcal{H}$ can be
considered as the space of random variables.

\begin{definition}
{\label{Def-1 copy(1)} {A \textbf{Sublinear expectation }\label{sube}
\index{Sublinear expectation}$\hat{\mathbb{E}}$ is a functional $\mathbb{{E} }:%
\mathcal{H}\rightarrow \mathbb{R}$ satisfying }}

\noindent{\textbf{\textup{(i)} Monotonicity: }%
\begin{equation*}
\hat{\mathbb{E}}[X]\geq \hat{\mathbb{E}}[Y] \ \
\text{if}\ X\geq Y,
\end{equation*}
}

\noindent \textbf{\textup{(ii)} Constant preserving:}
\begin{equation*}
{\hat{\mathbb{E}}}[c]=c\ \ \ \text{for}\ c\in \mathbb{R},
\end{equation*}

\noindent \textbf{\textup{(iii)} Sub-additivity: } For each $X,Y\in{\mathcal{%
H}}$,
\begin{equation*}
{\hat{\mathbb{E}}}[X+Y]\leq {\hat{\mathbb{E}}}[X]+{\hat{\mathbb{E}}}[Y],
\end{equation*}

\noindent \textbf{\textup{(iv)} Positive homogeneity:}
\begin{equation*}
{\hat{\mathbb{E}}}[\lambda X]=\lambda \hat{\mathbb{E}}[X]\ \ \ \text{for}\ \lambda \geq0%
\text{.}
\end{equation*}

The triple $(\Omega ,\mathcal{H},\hat{\mathbb{E}})$ is called a \textbf{%
sublinear expectation space}%
\index{Sublinear expectation space}.\label{subls}  
A sublinear expectation ${\hat{\mathbb{E}}}[\cdot]$ defined on   $(\Omega ,\mathcal{H})$
is said to be \textbf{regular}
%\footnote{\color{red} Is this regularity too strong? Image that even the linear expectation does not satisfy this regularity in general. 
%Shall we say the convergence hold for any uniformly bounded sequence $\{X_{i}\}$? }
 if 
\begin{equation}\label{1.2.1}
{\hat{\mathbb{E}}}[X_{i}]\rightarrow0
\end{equation}
for each sequence $\{X_{i}\}_{i=1}^{\infty}$ of random variables in  $\mathcal{H}$  such that\ $X_{i}(\omega)\downarrow0$ for each $\omega\in \Omega$.

If \textup{(i)} and
\textup{(ii)} are satisfied, {{$\hat{\mathbb{E}}$ is called a }}\textbf{nonlinear
expectation} and
\index{Nonlinear expectation} the triple $(\Omega ,\mathcal{H},\hat{\mathbb{E}}%
\mathbb{)}$ is called a \textbf{nonlinear expectation space}
\index{Nonlinear expectation space}.
\end{definition}

\begin{definition}\label{DefI.1.2}
Let $\hat{\mathbb{E}}_{1}$ and $\hat{\mathbb{E}}_{2}$ be two nonlinear expectations
defined on $(\Omega ,\mathcal{H})$. $\hat{\mathbb{E}}_{1}$ is said to be \textbf{%
dominated} by $\hat{\mathbb{E}}_{2}$ if
\begin{equation}
\hat{\mathbb{E}}_{1}[X]-\hat{\mathbb{E}}_{1}[Y]\leq \hat{\mathbb{E}}_{2}[X-Y]\ \ ~%
\text{for}\ X,Y\in \mathcal{H}.  \label{DominationI.1.2}
\end{equation}
\end{definition}

\begin{remark}
From \textup{(iii)}, a sublinear expectation is dominated by itself. In many
situations, \textup{(iii)} is also called the property of self-domination.
If the inequality in \textup{(iii)} becomes equality, then {{$\hat{\mathbb{E}}$}}
is a linear expectation, i.e., {{$\hat{\mathbb{E}}$}} is a linear functional
satisfying \textup{(i)} and \textup{(ii)}.
\end{remark}

\begin{remark}
\textup{(iii)+(iv)} is called \textbf{sublinearity}%
\index{Sublinearity}. This sublinearity implies

\noindent\textbf{\textup{(v)}} \textbf{Convexity}:
\begin{equation*}
\hat{\mathbb{E}}[\alpha X+(1-\alpha)Y]\leq \alpha \hat{\mathbb{E}}[X]+(1-\alpha)\mathbb{E%
}[Y]\ \
\text{for}\ \alpha \in \lbrack0,1].
\end{equation*}
If a nonlinear expectation $\hat{\mathbb{E}}$ satisfies convexity, we call it a
\textbf{convex expectation}.%
\index{Convex expectation}

The properties \textup{(ii)}\textup{+}\textup{(iii)} implies

\noindent\textup{\textbf{(vi)}} \textbf{Constant translatability}:
\begin{equation*}
\hat{\mathbb{E}}[X+c]=\hat{\mathbb{E}}[X]+c\ \
\text{for} \ c\in\mathbb{R}.
\end{equation*}
\end{remark}

In this paper, we are mainly  concerned with sublinear expectations.
In the following sections, unless otherwise stated, we consider the
following sublinear expectation space $(\Omega,\mathcal{H},\hat{\mathbb{E}}\mathbb{%
)}$: if $X_{1},\cdots,X_{n}\in \mathcal{H}$ then $\varphi(X_{1},%
\cdots,X_{n})\in \mathcal{H}$ for each $\varphi \in C_{Lip}(\mathbb{R}^{n})
$\label{cllip} where $C_{Lip}(\mathbb{R}^{n})$ denotes the linear space of
functions $\varphi$ satisfying the following Lipschitz condition:
\begin{align*}
|\varphi(x)-\varphi(y)| & \leq C|x-y|\,\,\,\,  \ \text{ for all }\
x,y\in \mathbb{R}^{n}\text{, \ } \\
\ & \text{ for some constant }C>0 \text{ depending on }\varphi.
\end{align*}
Often the $n$-dimensional  random variables  $X=(X_{1},\cdots,X_{n})$ is called an $n$-dimensional random
vector, denoted by $X\in \mathcal{H}^{n}$.

Here we mainly use $C_{Lip}(\mathbb{R}^{n})$ in our framework only for some
convenience of techniques. In fact our essential requirement is that $%
\mathcal{H}$ contains all constants and, moreover, $X\in \mathcal{H}$
implies $\left \vert X\right \vert \in \mathcal{H}$. In general, $C_{Lip}(%
\mathbb{R}^{n})$ can be replaced by any one of the following spaces of
functions defined on $\mathbb{R}^{n}$.

\begin{itemize}

\item {\ $L^{0}(\mathbb{R}^{n})$: the space of Borel measurable functions; \label{I1.L0Rn}}

\item {\ $\mathbb{L}^{\infty}(\mathbb{R}^{n})$: the space of bounded
Borel-measurable functions; }\label{I1.LinftyRn}%%

\item {\ $C_{b}(\mathbb{R}^{n})$: the space of bounded and continuous
functions;}\label{I1.CbRn}%%

\item {$C_{b}^{k}(\mathbb{R}^{n})$: the space of bounded and $k$-time
continuously differentiable functions with bounded derivatives of all orders
less than or equal to $k$;\label{cbk}}

\item {$C_{l,Lip}^{k}(\mathbb{R}^{n})$: the space of  $k$-time
continuously differentiable functions, whose partial derivatives of all orders
less than or equal to $k$ are in $C_{l,Lip}(\mathbb{R}^{n})$;}\label{clk}

\item {$C_{l.Lip}(\mathbb{R}^{n})$}: the linear space of functions $\varphi$ satisfying the following locally Lipschitz condition:\label{I1.ClLipRn}%%
\begin{align*}
|\varphi(x)-\varphi(y)| & \leq C(1+|x|^{m}+|y|^{m})|x-y|\ \ \text{for}\
x,y\in \mathbb{R}^{n}\text{, \ } \\
\ & \text{some }C>0\text{, }m\in \mathbb{N}\text{ depending on }\varphi.
\end{align*}

\item {\ $C_{b.Lip}(\mathbb{R}^{n})$: the space of bounded and Lipschitz
continuous functions; \label{cbl}}\label{I1.CbLipRn}%%

\item {\ $C_{unif}(\mathbb{R}^{n})$: the space of bounded and uniformly
continuous functions. \label{cu}}\label{I1.CunifRn}%%
\end{itemize}

% Observe that the above linear expectation $E_{\theta}$ is possibly  a finitely
%additive. We now give an important sufficient condition for the $\sigma$-additivity
%of  such $E_{\theta}$:

\begin{theorem} (Robust Daniell-Stone Theorem)
\label{r1} Assume that $(\Omega,\mathcal{H},{\hat{\mathbb{E}}} )$ is a %regular 
sublinear expectation  space. If   ${\hat{\mathbb{E}}}[\cdot]$ is regular, then there exists a class of probability measures $\{P_{\theta}\}_{\theta\in \Theta}$ on $(\Omega, \sigma(\mathcal{H}))$ such that 
\begin{equation}\label{1.2.2}
{\hat{\mathbb{E}}}[X]=\max_{\theta\in \Theta}\int_{\Omega}X(\omega)dP_{\theta},\ \  \textit{ for each }  X\in \mathcal{H}.
\end{equation}
where $\sigma({\mathcal{H}})$ is the smallest $\sigma$-algebra generated by $\mathcal{H}$.
\end{theorem}
 
In fact, the $\{P_{\theta}\}_{\theta\in\Theta}$ can be chosen as all probability measure $P_\theta$ dominated by $\hat{\mathbb{E}}$. The subset $\{P_{\theta}\}_{\theta\in\Theta}$ can be used to denote the uncertainty subset of probabilities.

\subsection{Distributions and independence of random variables}

We now give the notion of distributions of random variables under sublinear
expectations.
 
 Let $X=(X_{1},\cdots ,X_{n})$ be a given $n$-dimensional random vector on a
nonlinear expectation space $(\Omega ,\mathcal{H},\hat{\mathbb{E}})$. We define a
functional on $C_{Lip}(\mathbb{R}^{n})$ by
\begin{equation*}
\mathbb{F}_{X}[\varphi ]:=\hat{\mathbb{E}}[\varphi (X)]:\varphi \in C_{Lip}(%
\mathbb{R}^{n})\rightarrow \mathbb{R}.\end{equation*}
The triple $(\mathbb{R}^{n},C_{Lip}(\mathbb{R}^{n}),\mathbb{F}_{X})$ forms
a nonlinear expectation space. $\mathbb{F}_{X}$ is called the (nonlinear)  \textbf{%
distribution}
\index{Distribution}of $X$ under $\hat{\mathbb{E}}$. This notion is very useful
for a sublinear expectation space $\hat{\mathbb{E}}$. In this case $\mathbb{F}_{X}$
is also a sublinear expectation. Furthermore we can prove that (see Theorem %
\ref{r1}), there exists a family of probability measures $\{F_{X}(\theta,\cdot)\}_{\theta \in \Theta }$ defined on $(\mathbb{R}^{n},\mathcal{B}(%
\mathbb{R}^{n}))$ such that%
\begin{equation*}
\mathbb{F}_{X}[\varphi ]=\sup_{\theta \in \Theta }\int_{\mathbb{R}%
^{n}}\varphi (x)F_{X}^{\theta }(dx),\label{I3.Fx}\ \
\text{for each }\varphi \in C_{b.Lip}(\mathbb{R}^{n}).
\end{equation*}%
Thus $\mathbb{F}_{X}[\cdot ]$ characterizes the uncertainty of the
distributions of $X$.

\begin{definition}
\label{d1} Let $X_{1}$ and $X_{2}$ be two $n$--dimensional random vectors
defined on a {nonlinear expectation space }$(\Omega,\mathcal{H},%
\hat{\mathbb{E}})$. They are called \textbf{identically distributed},%
\index{Identically distributed}\label{iidd} denoted by $X_{1}\overset{d}{=}%
X_{2}$, if
\begin{equation*}
\hat{\mathbb{E}}[\varphi (X_{1})]=\hat{\mathbb{E}}[\varphi (X_{2})]\ \ \
\text{for}\ \varphi \in C_{Lip}(\mathbb{R}^{n}).
\end{equation*}%
It is clear that $X_{1}\overset{d}{=}X_{2}$ if and only if their
distributions coincide. We say that the distribution of $X_{1}$ is stronger
than that of $X_{2}$ if $\hat{\mathbb{E}}[\varphi (X_{1})]\geq \hat{\mathbb{E}}[\varphi (X_{2})]$, for each $\varphi \in C_{Lip}(\mathbb{R}^{n})$.
\end{definition}

\begin{remark}  
In the many cases of sublinear expectations, $X_{1}\overset{d}{=}X_{2}$ implies
that the uncertainty subsets of distributions of $X_{1}$ and $X_{2}$ are the
same, e.g., in the representation of uncertainty by  Theorem \ref{r1},
\begin{equation*}
\{F_{X_{1}}(\theta_{1},\cdot):\theta_{1}\in \Theta_{1}\}=\{F_{X_{2}}(\theta
_{2},\cdot):\theta_{2}\in \Theta_{2}\}.
\end{equation*}
Similarly if the distribution of $X_{1}$ is stronger than that of $X_{2}$,
then
\begin{equation*}
\{F_{X_{1}}(\theta_{1},\cdot):\theta_{1}\in
\Theta_{1}\}\supset\{F_{X_{2}}(\theta _{2},\cdot):\theta_{2}\in
\Theta_{2}\}.
\end{equation*}
\end{remark}

The distribution of $X\in \mathcal{H}$ has the following four typical
parameters:
\begin{equation*}
\bar{\mu}:={\hat{\mathbb{E}}}[X],\ \ \underline{\mu}:=-\hat{\mathbb{E}}[-X],\ \ \ \ \ \
\ \ \bar{\sigma}^{2}:={\hat{\mathbb{E}}}[X^{2}],\ \ \underline {\sigma}^{2}:=-{%
\hat{\mathbb{E}}}[-X^{2}].\ \
\end{equation*}
The intervals $[\underline{\mu},\bar{\mu}]$ and $[\underline{\sigma}^{2},%
\bar{\sigma}^{2}]$ characterize the \textbf{mean-uncertainty}%
\index{Mean-uncertainty} and the \textbf{2nd moment-uncertainty}
\index{Variance-uncertainty}of $X$ respectively.

\begin{lemma}\label{I-le3}
Let $(\Omega ,\mathcal{H},\hat{\mathbb{E}})$ be a regular sublinear expectation
space. Let $X\in \mathcal{H}^{d}$ be given.  Then for each sequence
$\{\varphi _{n}\}_{n=1}^{\infty }\subset C_{b.Lip}(\mathbb{R}^{d})$
satisfying $\varphi _{n}\downarrow 0$, we have
$\hat{\mathbb{E}}[\varphi _{n}(X)]\downarrow 0$.
\end{lemma}

\begin{lemma}\label{I-le4}
Let $(\Omega ,\mathcal{H},\hat{\mathbb{E}})$ be a sublinear expectation
space and let $\mathbb{F}_{X}[\varphi ]:=\hat{\mathbb{E}}[\varphi (X)]$ be
the sublinear distribution of $X\in \mathcal{H}^{d}$. Then there
exists a family of probability measures $\left\{ F_{\theta }\right\}
_{\theta \in \Theta }$ defined on
$(\mathbb{R}^{d},\mathcal{B}(\mathbb{R}^{d}))$ such that
\begin{equation}
\mathbb{F}_{X}[\varphi ]=\sup_{\theta \in \Theta }\int_{\mathbb{R}%
^{d}}\varphi (x)F_{\theta }(dx),\ \ \varphi \in
C_{b.Lip}(\mathbb{R}^{d}). \label{Distr}
\end{equation}
\end{lemma}

\begin{remark}
The above lemma tells us that in fact the sublinear distribution $\mathbb{F}%
_{X}$ of $X$ characterizes the uncertainty of distribution of $X$
which is an subset of distributions $\left\{ F_{\theta }\right\}
_{\theta \in \Theta } $.
\end{remark}

\begin{definition}
A sequence of $n$-dimensional random vectors $\left\{ \eta _{i}\right\}
_{i=1}^{\infty }$ defined on a nonlinear expectation space $(\Omega ,%
\mathcal{H},\hat{\mathbb{E}})$ is said to \textbf{converge in distribution}
\index{Converge in distribution}(or \textbf{converge in law})%
\index{Converge in law} under $\hat{\mathbb{E}}$ if for each $\varphi \in
C_{b.Lip}(\mathbb{R}^{n})$, the sequence $\left\{ \hat{\mathbb{E}}[\varphi (\eta
_{i})]\right\} _{i=1}^{\infty }$ converges.
\end{definition}

The following result is easy to check.

\begin{proposition}
Let $\left \{ \eta_{i}\right \} _{i=1}^{\infty}$ converge in law in the
above sense. Then the mapping $\mathbb{F}[\cdot]:C_{b.Lip}(\mathbb{R}%
^{n})\rightarrow \mathbb{R}$ defined by%
\begin{equation*}
\mathbb{F}[\varphi]:=\lim_{i\rightarrow \infty}\hat{\mathbb{E}}[\varphi(\eta_{i})]%
\ \
\text{for}\ \varphi \in C_{b.Lip}(\mathbb{R}^{n})
\end{equation*}
is a sublinear expectation defined on $(\mathbb{R}^{n},C_{b.Lip}(\mathbb{R}%
^{n}))$.
\end{proposition}

The following notion of independence plays a key role in the
nonlinear expectation theory.

\begin{definition}
\label{d2} In a nonlinear expectation space $(\Omega,\mathcal{H},\hat{\mathbb{E}})$%
, a random vector $Y\in \mathcal{H}^n$ is said to be \textbf{independent}
\index{Independent} of another random vector $X\in \mathcal{H}^m$ under $%
\hat{\mathbb{E}}[\cdot]$ if for each test function $\varphi \in C_{Lip}(\mathbb{R%
}^{m+n})$ we have
\begin{equation*}
\hat{\mathbb{E}}[\varphi(X,Y)]=\hat{\mathbb{E}}[\hat{\mathbb{E}}[\varphi(x,Y)]_{x=X}].
\end{equation*}
\end{definition}

\begin{remark}
{The situation \textquotedblleft$Y$ is independent of  $X$%
\textquotedblright often appears when $Y$ occurs after $X$, thus a robust
expectation should take the information of $X$ into account. }
\end{remark}

\begin{remark}
In a sublinear expectation space
$(\Omega,\mathcal{H},\hat{\mathbb{E}})$, $Y$ is
independent of  $X$ means that the uncertainty of distributions $%
\{F_{Y}(\theta,\cdot):\theta \in \Theta \}$ of $Y$ does not change after each
realization of $X=x$. In other words, the ``conditional sublinear
expectation'' of $Y$ with respect to $X$ is $\hat{\mathbb{E}}[\varphi(x,Y)]_{x=X}$%
. In the case of linear expectation, this notion of independence is just the
classical one.

It is important to observe that, under a nonlinear expectation, $Y$ is independent of $X$ does not automatically imply that $X$ is also independent of $Y$.
Examples can be found in Peng \cite{p7}.   Hu \cite{h1} systematically characterized the situation when $X$  and $Y$ are mutually independent. 
\end{remark}

The independence property of two random vectors $X,Y$ involves only the
``joint distribution'' of $(X,Y)$. We refer to Peng \cite{p7} for the existence  of 
random vectors $X_1,X_2,\cdots,X_n$, with given ``marginal distributions'' and with a
specific direction of independence. 

\subsection{Maximal distributions and law of large numbers }

Let us define a special type of very simple distributions which are
frequently used in practice, known as ``worst case risk measure''. 

\begin{definition}
(\textbf{Maximal distribution}) \label{Prop-G1 copy(1)}%
\index{Maximal distribution}A $d$-dimensional random vector $%
\eta=(\eta_{1},\cdots,\eta_{d})$ on a sublinear expectation space $(\Omega,%
\mathcal{H},\hat{\mathbb{E}})$ is called \textbf{maximal distributed} if there
exists a bounded, closed and convex subset $\Gamma \subset \mathbb{R}^{d}$
such that%
\begin{equation*}
\hat{\mathbb{E}}[\varphi(\eta)]=\max_{y\in \Gamma}\varphi(y),\,\,\, \varphi\in C_{l.Lip}({\mathbb{R}}^d).
\end{equation*}
For the case  $d=1$ we have $\Gamma=[\underline{\mu},\overline{\mu}]$, where $%
\overline{\mu}=\hat{\mathbb{E}}[\eta]\ \text{and} \ \underline{\mu}=-\mathbb{E }%
[-\eta]. $
The distribution of $\eta$ is given by
\begin{equation*}
\mathbb{F}_{\eta}[\varphi]=\hat{\mathbb{E}}[\varphi(\eta)]=\sup _{\underline{\mu}%
\leq y\leq \underline{\mu}}\varphi(y)\ \ \ \text{for} \ \varphi \in C_{l.Lip}(%
\mathbb{R}).
\end{equation*} 
We denote by $\eta\overset{d}{=}M_{[\underline{\mu},\overline{\mu}]}$. 

\end{definition}

\begin{remark}
Here $\Gamma$ gives the degree of uncertainty of $\eta$. It is easy to check
that this maximally distributed random vector $\eta$ satisfies
\begin{equation*}
a\eta+b%
\bar{\eta}\overset{d}{=}(a+b)\eta\ \ \ \text{for }a,b\geq0,\
\end{equation*}
where $\bar{\eta}$ is an independent copy of $\eta$. 
In fact this relation characterizes a maximal distribution (see  \cite{p7}). Maximal
distribution is also called ``worst case risk measure'' in finance. It is easy to check that if $\eta'$ and $eta$ are both maximally distributed and 
if  $\eta'$ is independent of $\eta$, then $\eta$ is also independent of $\eta'$.
\end{remark}

\begin{remark}

Hu \cite{h1} proved that  for two non-trivial random variables $X$ and $Y$ under a sublinear expectation space, namely the distribution of $X$ is not linear and $Y$ is not a constant , if $X$ is independent from $Y$ and $Y$ is independent from $X$, then $X$ and $Y$ must be maximally distributed.

\end{remark}

In the limit theory of sublinear expectations,  the  law of large numbers play a fundamental role. 

\begin{theorem}\label{LLNaa}
(\textbf{Law of large numbers, Peng \cite{p7}})%
\index{Law of large numbers} {Let $\left \{ Y_{i}\right \} _{i=1}^{\infty}$
be a sequence of }$\mathbb{R}^{d}$-valued random variables on a sublinear
expectation space{\ $($}$\Omega,${$\mathcal{H},\hat{\mathbb{E}})$}. We assume that
$Y_{i+1}\overset{d}{=}Y_{i}$ and $Y_{i+1}$ is independent from $%
\{Y_{1},\cdots,Y_{i}\}$  for each $i=1,2,\cdots$. We assume furthermore the following uniformly integrable condition:
\begin{equation}  \label{2.3.17}
\lim_{\lambda\to +\infty }{\hat{\mathbb{E}}}[(|Y_1|-\lambda)^+]=0.
\end{equation}
Then the sequence $\{(Y_1+\cdots+Y_n)/n\}_{n=1}^\infty$  converges in law to a maximal distribution, i.e., %
\begin{equation}
\lim_{n\rightarrow \infty}\hat{\mathbb{E}}[\varphi(\frac{1}{n}\sum_{i=1}^nY_{i})]=\max_{\theta\in \bar{\Theta}}\varphi(\theta),\ \   \label{e325}
\end{equation}
for all functions $\varphi \in C(\mathbb{R}^{d})$ satisfying linear growth
condition \textup{(}$|\varphi(x)|\leq C(1+|x|)$\textup{)}, where 
$\bar{\Theta}$ is the (unique) bounded, closed and convex subset of $\mathbb{R}^d$ satisfying 
\[
\max_{\theta\in\bar{\Theta}}\left\langle p,\theta\right\rangle =\hat{\mathbb{E}}%
[\left\langle p,Y_{1}\right\rangle ],\ \ \ p\in\mathbb{R}^{d}.
\]
\end{theorem}
\begin{remark}
The above LLN in the framework of sublinear expectation tells us that the maximal distribution can be widely applied in situations where probability uncertainty 
cannot be negligible. 
\end{remark}

\begin{remark} \cite{p3,p7} first provided the above LLN (as well as the corresponding central limit theorem) under a stronger moment condition on $Y_1$.  But the proof is easily 
adapted for the weaker condition (\ref{2.3.17}) proposed by \cite{zhang2}.  This condition is also equivalent to the following more classical one:
 \[
\lim_{n\to\infty}\bar{\mathbb{E}}[|Y_1|\textbf{1}_{|Y_1|>n}]=0.
\]
 Chen  \cite{CE1} then studied the corresponding strong LLN  by using Choquet capacity.  
We also refer to \cite{wuchen}
\cite{zhang2,zhang4,zhang1},  \cite{ba}, \cite{HZ}、\cite{LL1}  for further studies of limit theorems under PU. 
\end{remark}

The convergence result of (\ref{e325}) means that the sequence $\{\frac{1}{n}\sum_{i=1}^nY_{i}\}$ converges in law to a $d$-dimensional maximal distributed random vector $\eta$ and{\ the corresponding sublinear
function }$g:\mathbb{R}^{d}\rightarrow \mathbb{R}$ is defined by
\begin{equation*}
g(p):=\hat{\mathbb{E}}[\left \langle p,Y_{1}\right \rangle ]{,}\ \ p\in \mathbb{R}%
^{d}.
\end{equation*}
If all components of $Y_1=(Y_1^1,\cdots,Y_1^d)$ satisfy $\hat{\mathbb{E}}[Y_1^i]=-\hat{\mathbb{E}}[-Y_1^i] =\mu^i$, namely $Y_1$ has no mean uncertainty. Then we can check that $\eta\equiv \mu=(\mu^1,\cdots,\mu^d)$ and $\bar{\Theta}=\{\mu\}$.  In this case we can prove that $\{\frac{1}{n}\sum_{i=1}^nY_{i}\}$ converges strongly to the constant vector $\mu$ since, according to Theorem \ref{LLNaa}, 
\begin{equation}
\lim_{n\rightarrow \infty}\hat{\mathbb{E}}[|\frac{1}{n}\sum_{i=1}^nY_{i}-\mu|]=\max_{\theta\in \bar{\Theta}}|\theta-\mu|=0.
\end{equation}
But in many practical situations, the mean uncertainty is not negligible.     

%it is clear that 
%
%The advantage of  new law of large numbers tells us that, under many
%typical situations in which the uncertainty cannot neglectable, we need to take
%the risk caused by such uncertainty into a counter systematically, in many
%situations in which cumulations of uncertain quantities takes place, and thus
%people implicitly or explicitly apply the classical law of large numbers,
%i.e., in most part of insurance situations, in many case of  continuous time
%finance, we should re-examine carefully if a type of maximal distribution
%model need to be applied to estimate and evaluate possible future disasters.

\section{Main result}

In this section,  we present our main result.  Through the whole section,  let $X_{1},\cdots,X_{n}$ be $n$ copies of the same %i.i.d. sample of size $n$ from the population of 
maximal distribution 
\[
X_{i}\overset{d}{=}M_{[\underline{\mu},\overline{\mu}]},\  \  \  \ i=1,\cdots,n.
\]
with unknown parameters  $\underline{\mu}\leq \overline{\mu}$, and $X_{i}$ is independent of $\{X_{j}\}_{j=1,\cdots, i-1}$.  In short, we say our $X_{1},\cdots, X_{n}$ are i.i.d..  It is clear that maximum distribution
is completely determined by these two parameters.  It is then
important to construct statistics to estimate  parameters  $\underline{\mu}$ and $ \overline{\mu}$ properly.

It is natural to  think of the following statistics:
\begin{align}
\overline{\mu}^{\ast}(X_{1},\cdots,X_{n}) &  :=\max \{X_{1},\cdots
,X_{n}\},\label{1}\\
\underline{\mu}_{\ast}(X_{1},\cdots,X_{n}) &  :=\min \{X_{1},\cdots
,X_{n}\}.\label{2}%
\end{align}

Like in the classical statistical theory, we want our estimator to be unbiased 
in some sense, which should be redefined properly. 
%. estimator for $\underline{\mu}$ and $ \overline{\mu}$. We first present the definition of unbiasedness in nonlinear expectation. 

\begin{definition}
Let  $f_n \in C(\mathbb{R}^{n})$,   a statistic 
$T_n=f_{n}(X_1,\cdots,X_n)$ is called an unbiased estimator of  $\overline{\mu}$ (resp. for $\underline{\mu}%
$) if
\begin{align*}
\mathbb{\hat{E}}[f_{n}(X_{1},\cdots,X_{n})]  & =\overline{\mu}\  \text{(resp.
}-\mathbb{\hat{E}}[-f_{n}(X_{1},\cdots,X_{n})]=\underline{\mu}\text{),}\  \\
\text{for all }-\infty & <\underline{\mu}\leq \overline{\mu}<\infty \text{.
}\  \
\end{align*}

\end{definition}

We have the following lemma. 

\begin{lemma}
\label{lem2.1} If $f_n\in C_{Lip}(\mathbb{R}^{n})$,  the estimator $T_n=f_{n}(X_1,X_2,\cdots,X_n) $ is  unbiased for the upper mean
$\overline{\mu}$,  then for all $\underline{\mu}\leq \overline{\mu}$, we have,
\begin{align}
%\max_{1\leq i\leq n}
\max_{(x_{1},\cdots, x_{n})\in [\underline{\mu},\overline{\mu}]^{n}}%
f_{n}(x_{1},\cdots,x_{n})  & =\overline{\mu},\label{eq1}\\
\text{(resp. \ }%\min_{1\leq i\leq n}
\min_{(x_{1},\cdots, x_{n})\in [\underline{\mu},\overline{\mu}]^{n}}
%underline{\mu}\leq x_{i}\leq \overline{\mu}}
f_{n}(x_{1},\cdots,x_{n})  & =\underline{\mu}%
\text{)}\label{eq2}%
\end{align}
Consequently, for all $\underline{\mu}\leq \overline{\mu}$, and $(x_{1}%
,\cdots,x_{n})\in \lbrack \underline{\mu},\overline{\mu}]^{n}$,
\begin{align}
f_{n}(x_{1},\cdots,x_{n})  & \leq \overline{\mu},\  \label{eq3}\\
\text{(resp. \ }f_{n}(x_{1},\cdots,x_{n})  & \geq \underline{\mu}%
\text{)}.\label{eq4}%
\end{align}

\end{lemma}

\begin{proof}
Since
\begin{eqnarray*}
\hmE [f_{n}(X_{1},\cdots,X_{n})]  
&=& \hmE \left[\left \{  \hmE[\max_{\lmu \leq x_{n}\leq \bmu } f_{n}(x_{1},\cdots,x_{n})]\right \}  _{x_{1}=X_{1},\cdots,x_{n-1}=X_{n-1}}\right]  \\
&=&\hmE \left[ \max_{\lmu \leq x_{n}\leq \bmu } f_{n}(X_{1},\cdots,X_{n-1}, x_{n})\right]\\
& =& ...=\max_{1\leq i\leq n}\max_{\underline{\mu}\leq x_{i}\leq \overline{\mu}%
}f_{n}(x_{1},\cdots,x_{n}).
\end{eqnarray*}
Denote $f_{n-1}(x_{1}, \cdots, x_{n-1}):=\max_{\lmu \leq x_{n}\leq \bmu } f_{n}(x_{1},\cdots,x_{n-1}, x_{n})$, 
then 
\begin{eqnarray*}
&&|f_{n-1}(a_{1}, \cdots, a_{n-1})-f_{n-1}(b_{1},\cdots, b_{n-1})|\\
&\le&\max_{x_{n}\in [\lmu, \bmu]} |f_{n}(a_{1},\cdots, a_{n-1}, x_{n})-f_{n}(b_{1},\cdots, b_{n-1}, x_{n})|\\
&\le & \max_{x_{n}\in [\lmu, \bmu]}L|(a_{1},\cdots, a_{n-1})-(b_{1},\cdots, b_{n-1})|\\
&=&L|(a_{1},\cdots, a_{n-1})-(b_{1},\cdots, b_{n-1})|,
\end{eqnarray*}
which means $f_{n-1}\in C_{Lip}(\R^{n-1})$, hence we can continue our equality 
\begin{eqnarray*}
\hmE [f_{n}(X_{1},\cdots,X_{n})]  
&=& \hmE [f_{n-1}(X_{1},\cdots,X_{n-1})] \\
&\cdots&\\
&=& \hmE [f_{1}(X_{1})]\\
&=&\max_{x_{1}\in [\lmu, \bmu]}f_{1}(x_{1})\\
&=&\max_{x_{1}\in [\lmu, \bmu]}\max_{x_{2}\in [\lmu,\bmu]}f_{2}(x_{1}, x_{2})\\
&=&\max_{(x_{1},x_{2})\in [\lmu, \bmu]^{2}}f_{2}(x_{1}, x_{2})\\
&\cdots&\\
&=&  \max_{(x_{1},\cdots , x_{n})\in [\lmu, \bmu]^{n}}f_{n}(x_{1},\cdots, x_{n}).
\end{eqnarray*}

Then we  have (\ref{eq1}) and thus (\ref{eq3}). The proof of (\ref{eq2}) and (\ref{eq4}) are similar.
\end{proof}

Now, we present our main result.

\begin{theorem}
\label{thm2.2}Let  $X_{1},\cdots,X_{n}$ be i.i.d. sample of size $n$ from the population of maximal
distribution 

\[
X_{i}\overset{d}{=}M_{[\underline{\mu},\overline{\mu}]},\  \  \  \ i=1,\cdots,n,
\]
with unknown parameters  $\underline{\mu}\leq \overline{\mu}$.   Then we have, quasi surely (i.e., $P_{\theta}$-almost surely for any $\theta\in \Theta$),
\[
\underline{\mu}\leq \min \{X_{1}(\omega),\cdots,X_{n}(\omega)\} \leq \max
\{X_{1}(\omega),\cdots,X_{n}(\omega)\} \leq \overline{\mu}.
\]
Moreover,
\[
\widehat{\overline{\mu}}_{n}=\max \{X_{1},\cdots,X_{n}\}
\]
is the largest unbiased  estimator for the upper mean $\overline{\mu}$, 
\[
\widehat{\underline{\mu}}_{n}=\min \{X_{1},\cdots,X_{n}\}
\]
is the smallest  unbiased estimator for the lower mean $\underline{\mu}$.
\end{theorem}

\begin{proof}
It is easy to check that $\widehat{\overline{\mu}}=\max \{X_{1},\cdots,X_{n}\}$
is an unbiased estimator  for the unknown
upper mean $\overline{\mu}$ and  $\widehat{\underline{\mu}}=\min \{X_{1}%
,\cdots,X_{n}\}$ is an unbiased estimator for the unknown  lower mean $\underline{\mu}$.

Let $T_{n}=f_{n}(Y_{1},\cdots,Y_{n})$ be a given unbiased  estimator for the
upper mean $\overline{\mu}$.  For any  $y_{1},\cdots,y_{n}%
\in \mathbb{R}$, we  set
\[
\overline{\tilde\mu}=\max \{ y_{1},\cdots,y_{n}\},\  \  \  \underline{\tilde\mu
}=\min \{ y_{1},\cdots, y_{n}\},
\]
and consider the case $Y_{i}\overset{d}{=}M_{[\underline{\tilde\mu},\overline{\tilde\mu}%
]}$.  According to Lemma \ref{lem2.1}, the unbiased estimator $T_{n}=f_{n}(Y_{1},\cdots,Y_{n})$, $f_n$ must
satisfy (\ref{eq3}), namely,%
\[
f_{n}(y_{1},\cdots,y_{n})\leq \overline{\tilde \mu}=\max \{ y_{1},\cdots,y_{n}\}.
\]
Since $y_{1},\cdots,y_{n}$ can be arbitrarily chosen, we then
have
\[
f_{n}(y_{1},\cdots,y_{n})\leq \max \{y_{1},\cdots,y_{n}\},\  \  \forall
y_{1},\cdots,y_{n}\in \mathbb{R}.
\]
Thus $\widehat{\overline{\mu}}_n$ is the largest estimator for the upper mean.
We can prove that $\widehat{\underline{\mu}}_n$ is
the smallest estimator for the lower mean. 
\end{proof}

The next proposition tell us that the two estimators are both maximal distributed with the same lower
mean $\underline{\mu}$ and upper mean $\overline{\mu}$.

\begin{proposition}\label{prop1}
Let $X_{1},\cdots,X_{n}$ be a maximal distributed i.i.d. sample with
$X_{1}=M_{[\underline{\mu},\overline{\mu}]}$. Then
\begin{equation}
\max \{X_{1}(\omega),\cdots,X_{n}(\omega)\}\overset{d}{=}\min \{X_{1}(\omega),\cdots
,X_{n}(\omega)\}\overset{d}{=}M_{[\underline{\mu},\overline{\mu}%
]}.\  \  \label{eq5}%
\end{equation}
 \end{proposition}

\begin{proof}
It is clear that, for each $\varphi \in C(\mathbb{R)}$, we have
\begin{eqnarray*}
\mathbb{\hat{E}}[\varphi(\max \{X_{1},\cdots,X_{n}\})]  
& =&\max_{(x_{1},\cdots, x_{n})\in [\lmu, \bmu]^{n}} \varphi(x_{1}\vee x_{2} \vee \cdots \vee x_{n})\\
& =& \max_{x\in [\lmu, \bmu]} \varphi(x)=\hmE[\varphi(X_{1})].
\end{eqnarray*}
Thus $\max \{X_{1},\cdots,X_{n}\} \overset{d}{=}M_{[\underline{\mu}%
,\overline{\mu}]}$. Similarly we can prove that $\min \{X_{1},\cdots
,X_{n}\} \overset{d}{=}M_{[\underline{\mu},\overline{\mu}]}$. 
\end{proof}

\begin{remark}
In fact, we can prove that, for  continuous function $f\in C(\mathbb{R}%
^{n})$, we have
\begin{equation}
f(X_{1},\cdots,X_{n})\overset{d}{=}M_{[\underline{\mu}_{f}%
,\overline{\mu}_{f}]}\label{eq6}%
\end{equation}
where
\[
\overline{\mu}_{f}:=\max_{(x_{1},\cdots, x_{n})\in [\lmu, \bmu]^{n}}f(x_{1},\cdots,x_{n}),
\  \  \underline{\mu}_{f}:=\min_{(x_{1},\cdots, x_{n})\in [\lmu, \bmu]^{n}}f(x_{1},\cdots,x_{n}).
\]
Indeed, for each $\varphi \in C(\mathbb{R})$,
$$\hmE [\varphi(f(X_{1},\cdots,X_{n}))]   =\max_{(x_{1},\cdots, x_{n})\in [\lmu, \bmu]^{n}}\varphi(f(x_{1},\cdots,x_{n}))
 =\max_{y\in \lbrack \underline{\mu}_{f},\overline{\mu}_{f}]}\varphi(y),
$$
which implies (\ref{eq6}). 
\end{remark}

\section{General estimator}

In many practical situations,  it is not easy to get a maximal distributed i.i.d. sample $\{X_{1},\cdots,X_{n}\}$  such that   $X_{i}\overset{d}{=}M_{[\underline{\mu},\overline{\mu}]}$.    We usually treat an asymptotic case with  a sequence of samples
$\{X_{1}^{(k)},\cdots,X_{n}^{(k)}\}_{k=1}^{\infty}$, which converges in law to  $\{X_{1},\cdots,X_{n}\}$ , where $n$ is fixed.  In this situation,  we still can get 
asymptotically unbiased estimators of the upper mean $\overline{\mu}$ and the
lower mean $\underline{\mu}$.   Now, the definition of  asymptotically unbiased estimator is presented.

\begin{definition}
For a sequence of samples
$\{X_{1}^{(k)},\cdots,X_{n}^{(k)}\}_{k=1}^{\infty}$ and a function $f\in  C(\mathbb{R}^{n})$,  we call the statistic $f(X_{1}^{(k)},\cdots,X_{n}^{(k)})$ an   \textbf{asymptotically ubiased} estimator  
 \begin{equation}
\lim_{k\rightarrow \infty}\mathbb{\hat{E}}[(f(X_{1}^{(k)},\cdots,X_{n}%
^{(k)})]=\overline{\mu}\text{ (resp. }=\underline{\mu}\text{).}\label{eq7}%
\end{equation}  \end{definition}

Generally,  if there exists  a sequence of i.i.d. d-dimensional samples $\{X_{i}\}_{i=1}^{\infty}$, where $X_i\overset{d}{=}X$,  we can construct the  asymptotically largest unbiased estimator for the distribution of
$X$. To this end, let us take any test function $\varphi \in C_{b,Lip}(\mathbb{R}^{d})$. It is clear that $\{ \varphi(X_{i})\}_{i=1}^{\infty}$ is a
sequence of bounded and i.i.d. random variables. We define the statistic 
\[
M_{k,n}[\varphi]%=M_{k,n}[\varphi](X_{1},\cdots,X_{n})
:=\frac{1}{n}\sum_{i=1}^{n}\varphi(X_{{\rm trn}(i,k)}),
\]
where ${\rm trn}(n,k)=\frac{(n+k)(n+k-1)}{2}-(k-1)$ is the numbering of the 
$(n, k)$ node when we array integers  in the infinite matrix by triangle order. 

By Theorem \ref{LLNaa}, for each fixed
$k=1,2,\cdots$, $\{M_{k,n}[\varphi]\}_{n=1}^{\infty}$ converges in law to the maximal distribution
$M_{[\underline{\mu}_{X}(\varphi),\overline{\mu}_{X}(\varphi)]}$,  where
\[
\overline{\mu}_{X}(\varphi):=\hat{\mathbb{E}}[\varphi(X)],\  \  \  \underline{\mu}%
_{X}(\varphi)=-\hat{\mathbb{E}}[-\varphi(X)].
\]

Observe that $\overline{\mu}_{X}(\cdot):C_{b.Lip}(\mathbb{R}^{d}%
)\rightarrow \mathbb{R}$ is the sublinear distribution operator  of \ $X$.
Moreover, for each fixed $k$, the i.i.d  sequence
$\{M_{1,n}[\varphi],\cdots,M_{k,n}[\varphi]\}$ converges to 
$\{Y_{1},\cdots,Y_{k}\}$ weakly, where $Y_{k}\overset{d}{=}M_{[\underline{\mu}%
_{X}(\varphi),\overline{\mu}_{X}(\varphi)]}$. Then the estimator 
\[
\mathbb{T}_{k}[\varphi](X_{1},\cdots,X_{nk}):=M_{1,n}[\varphi]\vee \cdots \vee
M_{k,n}[\varphi]
\]
provide us the asymptotically largest unbiased  estimator of the sublinear
distribution of $X$  by Theorem \ref{thm2.2}.

In the statistic $M_{k,n}$, we needy to align samples in the triangle order, which 
is not an efficient way to use samples. In the real application, we can fix a large number $n$ based on the total sample size, and group all samples equally with some group size $n$, and then take $M_{k,n}$ as the average of $k^{th}$ group. 
\section*{Acknowledgment} 

We thank Dr. Wang Hanchao who  provided some very useful suggestions to improve the first draft of this paper.

 \end{document}